\documentclass[11pt,reqno]{amsart}
\usepackage{amsmath}
  \usepackage{paralist}
  \usepackage{graphics} 
  \usepackage{epsfig} 
 \usepackage[colorlinks=true]{hyperref}
\hypersetup{urlcolor=blue, citecolor=red}

\newtheorem{theorem}{Theorem}[section]

\newtheorem{lemma}[theorem]{Lemma}

\newtheorem{claim}{claim}
\newtheorem{conclusion}{conclusion}
\newtheorem{conjecture}{Conjecture}

\theoremstyle{definition}

\begin{document}
\begin{center}

\textbf{A Search of The Four-Color Theorem and its Higher Dimensional Generalization} \\
\centerline{\scshape Qizhi Wang }
\vspace{1pt}
School of Mathematical Sciences, Fudan University\\
Yangpu District, Shanghai 200433, China\\
qizhiwang12@fudan.edu.cn

\end{center}

\vspace{1pt}

\textbf{Abstract:} Four-Color Theorem has secret in its logical proof and actual operating. In this paper we will give a proof of Four-Color Theorem based on Kuratowski's Theorem using some induction argument and give a description of the most complicated coloring map, a simple proof of Kuratowski's Theorem using Euler charateristic is also presented. We also conjecture the higher dimensional generalization of Four-Color Theorem.\\



\section{Introduction}
In 1852, Francis Guthrie presented the well-known Four-Color Conjecture. Alfred Kempe gave his proof in 1879 but was spot an error and took a hundred years for the mathematicians to fix it. Though we have a computer-aided proof in 1976 by Appel and Haken and then a more explicit proof by Robertson et al. in 1996.([3])Also another proof in 2008 by Georges Gonthier.([1])  But we are still searching for a simpler proof.

For simplicity, we consider a connected map, without enclaves or exclaves, we can draw it in a plane as a graph, and we call it a planar graph. If four colors are sufficient to color the map such that no two regions with a common edge are colored with the same color, but it is permitted for two regions to be colored the same if they only have one or several vertexes in common, then we call the map is 4-colorable.

The Four-Color conjecture is the following theorem we want to prove.

\begin{theorem}
Any planar graph (map) is 4-colorable.
\end{theorem}

\section{Notation and Method}
As in the graph theory, we call the regions on the map faces. While the edge means the boundary line between two faces, the vertex means the point where three or more faces share on the boundary lines. So that we can treat the 4-color problem as a graph theory problem. In the following we mean a map the same as a graph.

The color we use will be denoted by $a, b, c, d, etc$. If n colors are enough to color the map, we call it n-colorable.

We denote $A\mapsto a$ to mean that $A$ is colored by the color $a$. If two faces $A,B$ share an edge we denote $A\leftrightarrow B$.

We will use Kuratowski's Theorem ([2]) in graph theory and then use some induction arguments to prove the Four-Color Theorem. We first give a simple topological proof which uses a method called polyhedral model.

\section{A Simple Proof of Kuratowski's Theorem for $K_{5}$}
The first thing we want to prove is that if a map has five faces, then it is 4-colorable. If this map is 5-colorable, then every two faces have one edge in common. In fact if we have to use five colors to color a map with five faces, then every two faces have to share an edge so as to distinguish them in two colors , to represent in dual graph, it is isomorphic to $K_{5}$.
But it is the famous Kuratowski's Theorem showing that a 5-colorable map won't exist for which has only five faces.

\begin{theorem}{(Kuratowski)}
A finite graph is planar if it does not contain a subgraph that is a subdivision of $K_{5}$ or $K_{3}^{3}$.
\end{theorem}

In this section, we will use a polyhedral cubic graph method instead of the planar graph. We call it cubic graph for brevity.

Take a cubic polyhedral as a rubber surface with empty intermediate, cut off its undersurface, then unfold it fully into the plane we will get a planar graph.

In this part we consider the face as connected open set. Although the shape of their corresponding faces change largely, the adjacent situation does not change at all. If two faces on the cubic graph share the same edge, the two regions on the planar graph will keep adjacent; if two faces on the cubic graph don't have the same edge, the regions on planar unfolding graph also won't have the same edge. So the transformation between cubic graph and planar unfolding graph will keep the adjacent situation between faces equivalent, if we want to study the law of planar graph coloring, we can instead study the law of cubic graph coloring. And then unfold it to show the coloring situation of the planar graph. And this method seems to give more insight for us.

First observe that if a cubic graph has five faces and it is  5-colorable, then the faces must be contiguous to each other. Otherwise it can be colored in less than 5 colors because we have conjugate faces on the graph.

Now we want to show such a cubic graph won't exist by contradiction.

\begin{lemma}(Euler formula)
For a cubic graph of our consideration, denote the number of its vertices, edges, and faces by $v, e, f$ respectively, we have the following formula by classical Euler formula (note that we have cut off the undersurface):
$v-e+f=1$.
\end{lemma}

If we suppose the five-faces 5-colorable cubic graph exist.

Now we can always construct the cubic graph in two steps:\\
Step 1: construct one cubic graph such that three of its faces is contiguous to its other four faces.\\
Step 2: if the remaining two faces don't be contiguous to each other, build new edge to make them be contiguous.

\begin{figure}[!h]\centering
  \includegraphics[width=0.5\textwidth]{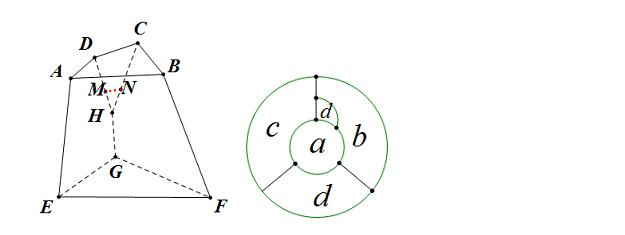}\\
  \caption{}\label{1}

\end{figure}

\begin{theorem}
The five-faces 5-colorable cubic graph won't exist.
\end{theorem}
\begin{proof}
Follow step 1, we see that the only graph we can construct is as figure 1, otherwise it can't be cubic graph.(The coloring situation is also shown.) In this cubic graph, we have five faces: pentagon $AEGHD, BFGHC$, tetragon $ABCD, ABFE$, triangle $CDH$.
But $\triangle CDH$ is not contiguous to $\square ABFE$, we must follow step 2.
As in figure 1, build a new boundary line $MN$ between $\triangle CDH$ and $\square ABFE$. Then we add 1 edge and 2 vertices. So we have
$v-e+f=(8+2)-(12+1)+5=2$.
Which contradicts lemma 3.1.
\end{proof}

\section{The Proof of Four-Color Theorem}
The principle of our proof is that we want to use induction argument, we first show a map with five faces is 4-colorable. Then we want to show that from five faces to six faces, four colors are also sufficient to color the map. And then we can apply the induction to prove the Four-Color Theorem.
So the nonexistence of $K_{5}$ makes using four colors is enough to color this map.

Start with these five faces $A,B,C,D,E$, from the last theorem the most complicated case is that there are only two faces that is not sharing an edge in common (say only D and E do not share an edge). We denote it by $D\parallel E$. For other cases, we can discuss in the same way and are much less restricted, so the other cases don't affect the possibility to color the map with four colors.

By assumption, as only different colors can be used if two faces are contagious. We can set the following:
$A\mapsto a, B\mapsto b, C\mapsto c, D\mapsto d, E\mapsto d$.

If we add a face $F$ which is near the other five faces, we want to show that four colors are enough to color the five faces plus face $F$. First do the coloring as above, as $F$ must fall in one of the following four classes, otherwise some contradiction will appear with the Kuratowski's Theorem.\\
1. $F\parallel A$, and we can do $F\mapsto a$.\\
2. $F\parallel B$, and we can do $F\mapsto b$.\\
3. $F\parallel C$, and we can do $F\mapsto c$.\\
4. $F\parallel D$, it must satisfy $F\parallel E$ and we can do $F\mapsto d$.\\

The first three classes are straightforward, for the last one, if $F\parallel D$ and $F\leftrightarrow E$, then $ABCEF$ is isomorphic to $K_{5}$ and it is a contradiction. If $F$ does not fall in these four cases, also we have $ABCEF$ is isomorphic to $K_{5}$ and it is a contradiction.

Up to this point, we have proved that if a map with 5 faces is 4-colorable, then a map has $(5+1)$ faces with the original 5 faces as its subgraph is 4-colorable. Now we use an induction argument as below.

\begin{claim}
 Assume a map with k faces is 4-colorable can induce that a map with $(k+1)$ faces is 4-colorable, we can prove that a map with $(k+1)$ faces is 4-colorable can induce that a map with $(k+2)$ faces is 4-colorable. Here $k\geq5$.
\end{claim}
In fact we can prove a stronger result. That is the following argument.

\begin{claim}
 Assume a graph with $(k+1)$ faces can be colored in four colors based on any coloring situation of one of its $k$ faces subgraph if any of its $k$ faces subgraph is 4-colorable. Then the same is true for $(k+2)$ faces graph.
\end{claim}

The proof of the above induction argument goes as follows.

If we have a map $M_{1}$ with $(k+2)$ faces, first choose a subgraph $M_{2}$ with $(k+1)$ faces such that the extra face $X$ is not contagious with one face $Y$ of the $(k+1)$ faces of $M_{2}$. Otherwise, it is easy to show that we can choose a subgraph that is isomorphic to $K_{5}$. So $M_{1}=M_{2}\bigsqcup{X}$ and $X\parallel Y$, where $Y\in M_{2}$. We can color $M_{2}$ in four colors by the assumption. Then basing on the coloring situation of $M_{2}-Y$, one can color the graph $(M_{2}-Y)\bigcup X$ which have $(k+1)$ faces. While $X\parallel Y$, this does not affect the coloring situation of $Y$. So we have finish the coloring of $M_{1}$.

Thus we finish the induction argument, so we have for any $k\geq5$, a map with $k$ faces is 4-colorable can induce that a map with $(k+1)$ faces is 4-colorable, as a map with five faces is 4-colorable from Kuratowski's Theorem, we get the consequence.

\section{Coloring Process for Complicated Maps}
Though from the last in induction theory it is believed that a map is always 4-colorable, we want to give an insightful proof of 4-color theorem. The principle is that for a map which have the most complicated contagious situation we can produce 4-color coloring. In the following graph, we use a dot on the circle to represent a face. If two faces are contagious we use full line to connect them, while two faces are not contagious we use dotted line to connect them.

As only different colors can be used if two faces are contagious. From the situation with 5 faces we can generalize it to the graph with n faces, the coloring situation for a map with most contagious faces is the following, where different letters means different colors, for example $A_{i}$ is colored by $A$.  The process is that first choose five faces that is most complicated in the last section, namely $A_{1}, B_{1}, C_{1}, D_{1},A_{2}$ from Kuratowski's Theorem, then keep on the process of coloring to the face nearby and we can produce the graph such that $n=i+j+k+l$, and it is indeed the most complicated graph we can draw.

\begin{figure}[!h]\centering
  \includegraphics[width=0.5\textwidth]{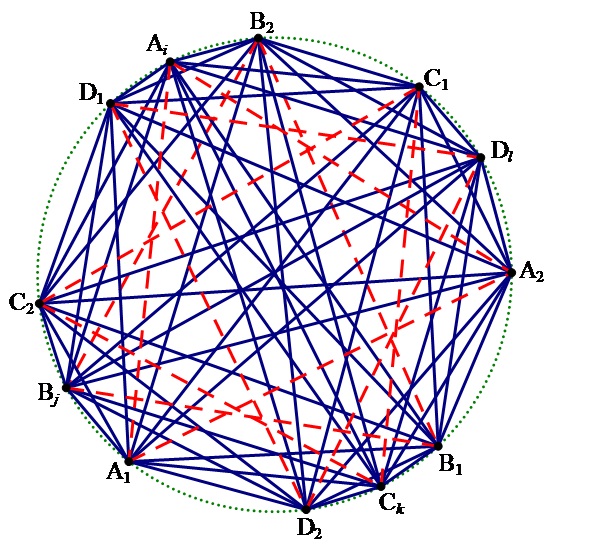}\\
  \caption{}\label{2}

\end{figure}

If this is not the map that have the most contagious faces, one should at least change one dotted line into a full line, which is always a contradiction for Kuratowski's Theorem as one can find five points that is isomorphic to $K_{5}$.

While for the actual situation, there can be much less restricted situation, and have more dotted lines, and so are always 4-colorable by this coloring map. This is why the computer plays a role for the different coloring situation for this complicated classification problem. While we just indicate the existence fact here.

\begin{conclusion}
 The more different boundary lines , the more complicated to color a map. While no matter how complicated the graph is, we always can produce a coloring map within four colors, and we do show how to do this and produce such a map.
\end{conclusion}

\section{The Higher Dimensional Conjecture}
With the development of society, the higher dimensional coloring problem will be useful in the future studying. For example, the 3-D typing technique, the star map in the universe and the 3-D cartoon is truly interesting such problems. With the help of topological and geometric methods, this coloring problem may create more problems.

The third section provides us with some insight that higher dimensional Kuratowski's Theorem may exist using the charateristic theory of higher dimension. Which means that for an n-dimensional graph with $(n+1)$-cells, it cannot be isomorphic to $K_{n+1}$.

So we can conjecture the following:
\begin{conjecture}
Any n-dimensional graph is $(n+2)$-colorable.
\end{conjecture}

This conjecture is reasonable, for $n=2$ it is just the Four-Color Theorem. For $n=1$, it is obvious that three colors are enough to separate the line segments in a curve. For the higher dimensional case, it is believed that we can follow the way of $n=2$ case to prove this conjecture but seems much more work for the associated Kuratowski's Theorem.

\section*{Acknowledgments} The author would like to thank his family for supporting him in the study especially his father.

{}
\end{document}